\documentclass[smallextended]{svjour3}       
\smartqed  
\usepackage[utf8]{inputenc}
\usepackage{hyperref}

\begin{document}
\title{On calculation of the interweight distribution\\ 
       of an equitable partition%
        \thanks{The paper has been published in the Journal of Algebraic Combinatorics, \url{http://dx.doi.org/10.1007/s10801-013-0492-3};
        the final publication is available at http://link.springer.com. \\
            The research was partially supported by 
            the Russian Foundation for Basic Research  (grant 13-01-00463),
            by the Ministry of Education and Science of Russian Federation
            (project 8227), and by the Target program of SB RAS for 2012-2014
            (integration project No. 14).}
       }
\author{Denis S. Krotov}
\institute{D. Krotov \at
              Sobolev Institute of Mathematics, 
  pr. Akademika Koptyuga 4, 
  Novosibirsk 630090, Russia. \\
  Novosibirsk State University, Pirogova 2, 
  Novosibirsk 630090, Russia. \\
              Tel.: +7-913-9331414 \\
              \email{krotov@math.nsc.ru}  
}

\date{Received: date / Accepted: date}

\maketitle

\begin{abstract}
We derive recursive and direct formulas
for the interweight distribution
of an equitable partition of a hypercube.
The formulas involve a three-variable
generalization of the Krawtchouk polynomials.

\end{abstract}

\section{Introduction}\label{s:0}
 We study the equitable partitions, 
   see e.g. \cite[\S5.1]{Godsil93}
(also known as regular partitions, 
   see e.g. \cite[\S11.1.B]{Brouwer}, 
partition designs, see e.g. \cite{CCD:92}, 
or perfect colorings, see e.g. \cite{FDF:PerfCol}; 
less popular equivalent terms include
coherent partitions \cite{Higman:75}, 
feasible colorations \cite[\S4.1]{CDS}, 
distributive colorings \cite{Vizing:95}),
of the $n$-cubes. 
The goal of the paper is to derive recursive 
and direct formulas
for the interweight distributions
of an equitable partition of a hypercube.
Some results are formulated in terms 
of the triangle distribution of the partition,
which relates to the interweight distributions
in the similar manner as the distance distribution relates to the weight distribution
(the later two concepts are well known in coding theory, see e.g. \cite{MWS}).
The formulas can be used to prove the nonexistence of equitable partitions with 
certain parameters, which is demonstrated by examples 
(Proposition~\ref{p:100}, Example~\ref{ex:22}).
The results are applicable
to the completely regular sets 
(including perfect codes, nearly perfect codes, 
and some kind of uniformly packed codes),
as they can be represented in terms of equitable partitions 
(see e.g. \cite{CCD:92}).
As well, the interweight distributions have a potential in
studying related objects, 
such as 
difference sets, which compose check 
matrices of completely regular codes,
and 
linear two-weight codes, 
which are dual to completely regular codes
(see e.g. \cite[Corollary~4.3]{CalKan:86:2weight}).

In Section~\ref{s:pre}, we give main definitions.
Section~\ref{s:WT} contains recursive formulas for the calculation
of the interweight distribution of an equitable partition of an $n$-cube.
One of the formulas is proved in Section~\ref{s:proof}, 
while the other are just simple corollaries of the first one.
In Section~\ref{s:poly}, a direct formula for 
the interweight distribution is derived, 
in terms of polynomials
in the quotient matrix of the equitable partition 
and their generating function.
In Section~\ref{s:ci}, we observe an empiric relation
with the so-called correlation-immunity bound on the quotient matrices
of equitable $2$-partitions of an $n$-cube.
In Section~\ref{s:real}, we briefly discuss a real-valued generalization 
of the equitable partitions. In the concluding section,
we give final remarks and formulate some open questions.

\section{Preliminaries}\label{s:pre}

An \emph{equitable partition}
of a graph $G=(V(G),E(G))$ 
is an ordered partition $C=(C_1,\ldots,C_m)$ 
of $V(G)$ such that 
for every $i$ and $j$ from $1$ to $m$ 
and every vertex $v$ from $C_i$ the number $S_{ij}$ 
of its neighbors from $C_j$ depends only on $i$ and $j$ 
and does not depend on the choice of $v$.
The matrix $S=(S_{ij})_{i,j=1}^m$ 
is called the \emph{quotient matrix} of $C$.

Let $C=(C_1,\ldots,C_m)$ 
be a collection of vertex sets 
of a graph $G$ of diameter $d$.
The \emph{weight distribution} of $C$ 
with respect to a vertex $v$ of $G$ 
is the collection of numbers 
$(W^r_{v,j})_{j=1}^{m}\vphantom{()}_{r=0}^{d}$ 
where $W^r_{v,j}$ is the number of vertices of $C_j$ 
at distance $r$ from $v$
(by the distance, we mean the natural graph distance, 
i.e., the length of a shortest path between two vertices).

One of well-known properties 
of the equitable partitions 
of distance regular graphs is the 
\emph{distance invariance}.
A collection $C=(C_1,\ldots,C_m)$ of mutually disjoint
vertex sets is called \emph{distance invariant} 
if for every $i\in \{1,\ldots,m\}$ the weight distribution 
of $C$
with respect to a vertex $v$ from $C_i$ 
does not depend on the choice of $v$ and
depends only on $i$.

\begin{remark}
The concept of the distance invariance,
defined as above,
is applicable to partitions, 
as well as to single sets, $C=(C_1)$.
In the last case, $C_1$ is known in coding theory 
as a \emph{distance invariant set}, 
or a \emph{distance invariant code} 
\cite{DelFar:89}.
\end{remark}

The weight distribution 
of an equitable partition $C$ 
of a distance regular graph
can be calculated
using recursive relations or the direct formula
(see \cite{Martin:PhD,Kro:struct})
$W_v^w = \Pi^{(w)}(S) W_v^0$   
where $W_v^w=(W^w_{v,1},\ldots,W^w_{v,m})$,
$S$ is the quotient matrix of $C$,
and $\Pi^{(w)}$ is a polynomial related to the graph:
if $A$ is the adjacency matrix of the graph
then $A^{(w)}=\Pi^{(w)}(A)$ is the distance-$w$ matrix
($A^{(w)}_{u,v}=1$
if the distance between the vertices $u$, $v$ equals $w$,
and $A^{(w)}_{u,v}=0$ otherwise).
In the current paper, we consider only $n$-cubes, 
which will be defined below; 
for the background on the distance regular graphs 
in general see, e.g., \cite{Brouwer}.

One of the known strengthenings 
of the distance invariance property
is the \emph{strong distance invariance} \cite{Vas2002en,Vas09:inter}.
For a fixed vertex $v$ from $C_i$, let $W^{r_1,r_2,r_3}_{ijk}$ 
denote the number of the pairs $(x,y)$ 
such that $d(x,y)=r_2+r_3$, $d(v,y)=r_1+r_3$, $d(v,x)=r_1+r_2$,  
$x\in C_j$, and $y\in C_k$.
A collection $C=(C_1,\ldots,C_m)$ of mutually disjoint
vertex sets is called \emph{strongly distance invariant} 
if its \emph{interweight distribution} 
$W_i=(W^{r_1,r_2,r_3}_{ijk})_{r_1,r_2,r_3=0}^d\vphantom{()}_{j,k=1}^{m}$ 
 with respect to a vertex $v$ from $C_i$ 
does not depend on the choice of $v$ 
(originally, elements of the interweight 
distribution were indexed by
the distances $a=d(x,y)$, $b=d(v,y)$, $c=d(v,x)$; 
by the reasons that can be seen from formulas below,
we reenumerate them using the indices 
$r_1=(-a+b+c)/2$, $r_2=(a-b+c)/2$, $r_3=(a+b-c)/2$).

The \emph{$n$-cube} $H^n$ is the graph $(V(H^n),E(H^n))$
whose vertex set is the set of all $n$-words 
in the alphabet $\{0,1\}$,
two words being adjacent if and only if 
they differ in exactly one position.

\begin{theorem}[\cite{Vas09:inter}]\label{th:Vas}
 The equitable partitions of the $n$-cubes 
 are strongly distance invariant.
\end{theorem}
The statement does not hold 
for distance regular graphs 
in general, 
see examples in \cite{Kro:struct}.

We will show how to calculate 
the interweight distribution 
of an equitable partition of an $n$-cube.
To formulate some of the results, 
we need to introduce a new
notion and the corresponding concept,
whose usability is briefly 
discussed in the
beginning of the next section.
Let $T^{r_1,r_2,r_3}_{ijk}$ 
denote the number of the triples $(v,x,y)$ 
such that $d(x,y)=r_2+r_3$, $d(v,y)=r_1+r_3$, $d(v,x)=r_1+r_2$, 
$v\in C_i$, $x\in C_j$, and $y\in C_k$.
We will refer to the collection of $T^{r_1,r_2,r_3}_{ijk}$ 
for all $i$, $j$, $k$, $r_1$, $r_2$, $r_3$ as the \emph{triangle distribution}
of $C$ (which can be, in this definition, 
an arbitrary family of subsets of $V(H^n)$).
Note that $W^{r_1,r_2,r_3}_{i j k} = T^{r_1,r_2,r_3}_{i j k} / |C_i| $ holds, 
due to the strong distance invariance of the equitable partitions. 
Therefore the interweight distribution can be easily calculated from the triangle distribution.
The array $T^{r_1,r_2,r_3}$ 
(and, similarly, $W^{r_1,r_2,r_3}$) will be treated 
as a row-vector of length $m^3$ whose elements 
$T^{r_1,r_2,r_3}_{ijk}$ ($W^{r_1,r_2,r_3}_{ijk}$, respectively) are indexed
by the triple $i,j,k\in \{1,\ldots,m\}$.
Next, we define three $m^3 \times m^3$ matrices 
$S'$, $S''$, $S'''$ such that 
multiplication of a row-vector 
to the matrix $S'$ 
($S''$, $S'''$, respectively) is the same 
as multiplication of the corresponding
three-indexed array 
to the matrix $S$ in the first 
(second, third, respectively) index.
That is, for $U=(U_{ijk})$ 
and $V=(V_{ijk})$:
\begin{eqnarray*}
V=US' & \Longleftrightarrow & 
V_{ijk} = \sum_{t=1}^m U_{tjk} S_{ti}
\quad\mbox{for all }i,j,k
\in\{1,\ldots,m\}, \\
V=US'' & \Longleftrightarrow & 
V_{ijk} = \sum_{t=1}^m U_{itk} S_{tj}
\quad\mbox{for all }i,j,k
\in\{1,\ldots,m\}, \\
V=US''' & \Longleftrightarrow & 
V_{ijk} = \sum_{t=1}^m U_{ijt} S_{tk}
\quad\mbox{for all }i,j,k
\in\{1,\ldots,m\}.
\end{eqnarray*}
Formally, 
$S'=S\otimes I \otimes I$, 
$S''=I \otimes S\otimes I$, 
$S'''=I\otimes I \otimes S$,  
where $I$ is the identity $m\times m$ matrix 
and $U=X \otimes Y\otimes Z$ 
denotes the $m^3 \times m^3$ matrix
with elements 
$U_{ijk,i'j'k'} = X_{i,i'}Y_{j,j'}Z_{k,k'}$.
We also define the diagonal matrix $D'$ such that
\begin{equation}\label{eq:D}
T^{r_1,r_2,r_3} = W^{r_1,r_2,r_3} D'.
\end{equation}
Formally, $D'=D\otimes I \otimes I$, where $D_{ii}=|C_i|$.

\begin{example}\label{ex:S}
If $S=\left(\begin{array}{cc} 0&3\\1&2\end{array}\right)$
and the elements of a vector are arranged as $$U=(U_{111},U_{112},U_{121},U_{122},U_{211},U_{212},U_{221},U_{222}),$$ then
$$
D'=\left(
\begin{array}{cc|cc||cc|cc}
2&0&0&0&0&0&0&0 \\
0&2&0&0&0&0&0&0 \\\hline
0&0&2&0&0&0&0&0 \\
0&0&0&2&0&0&0&0 \\\hline\hline
0&0&0&0&6&0&0&0 \\
0&0&0&0&0&6&0&0 \\\hline
0&0&0&0&0&0&6&0 \\
0&0&0&0&0&0&0&6
\end{array}
\right),
\qquad
S'=\left(
\begin{array}{cc|cc||cc|cc}
0&0&0&0&3&0&0&0 \\
0&0&0&0&0&3&0&0 \\\hline
0&0&0&0&0&0&3&0 \\
0&0&0&0&0&0&0&3 \\\hline\hline
1&0&0&0&2&0&0&0 \\
0&1&0&0&0&2&0&0 \\\hline
0&0&1&0&0&0&2&0 \\
0&0&0&1&0&0&0&2
\end{array}
\right),
$$ 
$$
S''=\left(
\begin{array}{cc|cc||cc|cc}
0&0&3&0&0&0&0&0 \\
0&0&0&3&0&0&0&0 \\\hline
1&0&2&0&0&0&0&0 \\
0&1&0&2&0&0&0&0 \\ \hline\hline
0&0&0&0&0&0&3&0 \\
0&0&0&0&0&0&0&3 \\\hline
0&0&0&0&1&0&2&0 \\
0&0&0&0&0&1&0&2
\end{array}
\right),
\qquad
S'''=\left(
\begin{array}{cc|cc||cc|cc}
0&3&0&0&0&0&0&0 \\
1&2&0&0&0&0&0&0 \\\hline
0&0&0&3&0&0&0&0 \\
0&0&1&2&0&0&0&0 \\\hline\hline
0&0&0&0&0&3&0&0 \\
0&0&0&0&1&2&0&0 \\\hline
0&0&0&0&0&0&0&3 \\
0&0&0&0&0&0&1&2
\end{array}
\right).
$$
\end{example}

\begin{lemma}\label{l:commut}
The matrices $S'$, $S''$, 
and $S'''$ commute with each other. 
The matrix $D'$ commutes with
$S''$ and $S'''$.
\end{lemma}
%

\begin{remark}\label{r:noncom}
The matrices $D'$ and $S'$ do not commute in general, see Example~\ref{ex:S}.
\end{remark}

\section{Recursive formulas}\label{s:WT}
In the following theorem, 
we present two groups of formulas.
The first three equations, (\ref{eq:W3})--(\ref{eq:W1}),
give recursions for $W^{r_1,r_2,r_3}$;
 they give more information than formulas 
 (\ref{eq:T3})--(\ref{eq:T1}), which concern to $T^{r_1,r_2,r_3}$, 
because $W^{r_1,r_2,r_3}$ are more refined 
characteristics than $T^{r_1,r_2,r_3}$.
On the other hand,
the system of equations (\ref{eq:T3})--(\ref{eq:T1})
is symmetric with respect to all three parameters $r_1$, $r_2$, $r_3$,
and we need this symmetry to derive one of the equations from the others.
The symmetry is a key point of the proof 
and justifies the use of the triangle distribution. 
Without this trick (or a separate combinatorial proof 
of the third equation for $W^{r_1,r_2,r_3}$,
which is expected to be complicated),
we have only two formulas for $W^{r_1,r_2,r_3}$, which are sufficient to 
calculate the interweight distributions 
from the weight distributions recursively,
but not for deriving the direct formulas (Section~\ref{s:poly}).

To understand the formulas below, 
it should be noted that, 
as it follows from definitions, 
the elements of $W^{r_1,r_2,r_3}$ 
and $T^{r_1,r_2,r_3}$ are zeros
if $r_1<0$, $r_2<0$, $r_3<0$, or $r_1+r_2+r_3>n$.

\begin{theorem} \label{th:W}
Let $C=(C_1,\ldots,C_m)$ be an equitable partition 
of an $n$-cube with quotient matrix $S$. 
Then the vectors $W^{r_1,r_2,r_3}$
satisfy the following equations:
\begin{eqnarray}
W^{r_1,r_2,r_3} S''' 
  & = & (r_1+1)W^{r_1+1,r_2-1,r_3} + (r_2+1)W^{r_1-1,r_2+1,r_3} 
  \nonumber\\ 
  &&{} + (n-r_1-r_2-r_3+1)W^{r_1,r_2,r_3-1}
       + \underline{(r_3+1)W^{r_1,r_2,r_3+1}}.
       \nonumber\\    \label{eq:W3}\\
W^{r_1,r_2,r_3} S'' 
  & = & (r_1+1)W^{r_1+1,r_2,r_3-1} + (r_3+1)W^{r_1-1,r_2,r_3+1}  
  \nonumber\\ 
  &&{} + (n-r_1-r_2-r_3+1)W^{r_1,r_2-1,r_3} 
       + \underline{(r_2+1)W^{r_1,r_2+1,r_3}},
      \nonumber\\    \label{eq:W2}\\
W^{r_1,r_2,r_3} S'^{\mathrm{T}}& = & 
(r_2+1)W^{r_1,r_2+1,r_3-1} + 
(r_3+1)W^{r_1,r_2-1,r_3+1}  \nonumber\\ &&{} + 
(n-r_1-r_2-r_3+1)W^{r_1-1,r_2,r_3} + 
\underline{(r_1+1)W^{r_1+1,r_2,r_3}}.
\nonumber\\     \label{eq:W1}
\end{eqnarray}
The triangle distribution
satisfies the following equations:
\begin{eqnarray}
T^{r_1,r_2,r_3} S''' & = & 
(r_1+1)T^{r_1+1,r_2-1,r_3} + 
(r_2+1)T^{r_1-1,r_2+1,r_3} \nonumber\\ &&{} + 
(n-r_1-r_2-r_3+1)T^{r_1,r_2,r_3-1} + 
\underline{(r_3+1)T^{r_1,r_2,r_3+1}},
\nonumber\\     \label{eq:T3}\\
T^{r_1,r_2,r_3} S'' & = & 
(r_1+1)T^{r_1+1,r_2,r_3-1} + 
(r_3+1)T^{r_1-1,r_2,r_3+1}  \nonumber\\ &&{} +  
(n-r_1-r_2-r_3+1)T^{r_1,r_2-1,r_3} + 
\underline{(r_2+1)T^{r_1,r_2+1,r_3}},
\nonumber\\     \label{eq:T2}\\
T^{r_1,r_2,r_3}  S' & = & 
(r_2+1)T^{r_1,r_2+1,r_3-1} + 
(r_3+1)T^{r_1,r_2-1,r_3+1}  \nonumber\\ &&{} + 
(n-r_1-r_2-r_3+1)T^{r_1-1,r_2,r_3} + 
\underline{(r_1+1)T^{r_1+1,r_2,r_3}}.
\nonumber\\     \label{eq:T1}
\end{eqnarray}
\end{theorem}
\begin{proof}
We will prove equation (\ref{eq:W3}) separately, in Section~\ref{s:proof}.
Equation (\ref{eq:W2}) can be obtained in the same manner.
Since 
$W^{r_1,r_2,r_3}D'=T^{r_1,r_2,r_3}$ 
and $D'$ commutes with $S''$ and $S'''$ (Lemma~\ref{l:commut}), 
we see that  (\ref{eq:T3}) and (\ref{eq:T2}) 
are straightforward from  (\ref{eq:W3}) and (\ref{eq:W2}).
By analogy, (\ref{eq:T1}) holds too (indeed, 
$T^{r_1,r_2,r_3}_{ijk}=T^{r_2,r_1,r_3}_{jik}$, by definition).
Using again
$W^{r_1,r_2,r_3}D'=T^{r_1,r_2,r_3}$ 
and noting that $D' S' D'^{-1}=S'^{\mathrm{T}}$
(the last follows from $D S D^{-1}=S^{\mathrm{T}}$, 
which is the same as $|C_i|S_{i,j}=|C_j|S_{j,i}$, $i,j=1,\ldots,m$), 
we derive (\ref{eq:W1}) from (\ref{eq:T1}).
\qed\end{proof}

The collection 
$(W^{r_1,r_2,r_3})_{r_1,r_2,r_3=0}^d
=((W^{r_1,r_2,r_3}_{ijk})_{i,j,k=1}^{m})_{r_1,r_2,r_3=0}^d$, 
can be treated as $m$ interweight distributions, 
accordingly with different values of $i$.
For a vertex $v$, the interweight 
distribution of $C$ with respect to $v$ is given
by the values $W^{r_1,r_2,r_3}_{ijk}$ with fixed $i$ such that $v\in C_i$. 
By the definition, 
$W^{0,0,0}_{iii}=1$, $i=1,\ldots,m$, 
and the other entries of $W^{0,0,0}$ are zeros.
Formulas (\ref{eq:W3})--(\ref{eq:W1}) express $W^{l_1,l_2,l_3}$ 
(in the underlined parts of formulas) 
as a combination of arrays 
with smaller index sum $l_1+l_2+l_3$; 
so, all the values are calculated recursively.
The situation with the triangle distribution
is similar with the only difference in the initial values:
$T^{0,0,0}_{iii}=|C_i|$, $i=1,\ldots,m$.
(Note also that, because of the obvious relations 
$|C_i|S_{i,j}=|C_j|S_{j,i}$ and $|C_1|+\ldots+|C_m|=2^n$,
the values $|C_i|$, $i=1,\ldots,m$, 
are derived from the quotient matrix $S$.)

\section{A proof of the recursion}\label{s:proof}
Before proving (\ref{eq:W3}),
we define two auxiliary concepts.

Given a collection $C=(C_{1},\ldots,C_{m})$
of subsets of the vertex set of a graph,
the \emph{spectrum} of a vertex set $X$ with respect to $C$
is the $k$-tuple ${\mathrm{Sp}_C}(X)=(x_1,\ldots,x_m)$, 
where $x_i = |X\cap C_{i}|$ 
(intuitively, we can treat $C_{1}$, \ldots, $C_{m}$ 
as colors and think about the color spectrum).
If $X$ is a multiset, 
then $x_i$ is defined as the sum over $C_{i}$ 
of the multiplicities in $X$.

The \emph{multi-neighborhood} $[X]$ of a vertex set $X$ 
is a multiset of vertices of the graph, 
where the multiplicity of a vertex 
is calculated as the number of its neighbors from $X$.
In other words, $[X]=\uplus_{x\in X}[x]$ 
where $[x]$ is the neighborhood of the vertex $x$
and $\uplus$ is the multiset union.

\begin{lemma}\label{l:spe}
 Let $C$ be an equitable partition 
 (of an arbitrary graph)
 with quotient matrix $S$.
 For every vertex set $X$,
 $$ {\mathrm{Sp}_C}([X]) = {\mathrm{Sp}_C}(X) \cdot S.$$
\end{lemma}
\begin{proof}
 By the definitions of an equitable partition, the multi-neighborhood, and the spectrum, we have 
 ${\mathrm{Sp}_C}([x]) = {\mathrm{Sp}_C}(\{x\})\cdot S$ for every vertex $x$.
Then,
 $$ {\mathrm{Sp}_C}([X]) 
 = {\mathrm{Sp}_C}(\uplus_{x\in X}[x]) 
 = \sum_{x\in X} {\mathrm{Sp}_C}([x]) 
 = \sum_{x\in X} {\mathrm{Sp}_C}(\{x\})\cdot S 
 = {\mathrm{Sp}_C}(X) \cdot S.$$
\qed\end{proof}

Now, we are ready to prove (\ref{eq:W3}).
For fixed vertices $v$ and $x$ of the $n$-cube, 
denote by $H_{v,x}^{r_1,r_2,r_3}$ the set of vertices $y$ such that
$d(v,y) = r_1+r_3$, $d(x,y)=r_2+r_3$, $d(v,x)=r_1+r_2$
(we do not restrict the values 
of the parameters $r_1$, $r_2$, $r_3$,
but note that by the definition $H_{v,x}^{r_1,r_2,r_3}$ 
is nonempty only for nonnegative $r_1$, $r_2$, $r_3$ 
satisfying $r_1+r_2=d(v,x)$ and $r_1+r_2+r_3 \leq n$).

Every vertex of $H_{v,x}^{r_1,r_2,r_3}$ has 
$r_1$ neighbors from $H_{v,x}^{r_1-1,r_2+1,r_3}$,
$r_2$ neighbors from $H_{v,x}^{r_1+1,r_2-1,r_3}$,
$r_3$ neighbors from $H_{v,x}^{r_1,r_2,r_3-1}$,
$n-r_1-r_2-r_3$ neighbors from $H^{r_1,r_2,r_3+1}$,
and no other neighbors
(to see this, we can consider without loss of generality that 
$v=0^n$, $x=1^{r_1+r_2}0^{n-r_1-r_2}$, 
$y=0^{r_2} 1^{r_1+r_3} 0^{n-r_1-r_2-r_3}$).

By the definition of the multi-neighborhood, we have
$$
[H_{v,x}^{r_1,r_2,r_3}]  
=  l_1 H_{v,x}^{r_1+1,r_2-1,r_3} 
\uplus l_2 H_{v,x}^{r_1-1,r_2+1,r_3}  
\uplus l_3 H_{v,x}^{r_1,r_2,r_3+1}
\uplus l_0 H_{v,x}^{r_1,r_2,r_3-1}
$$
where 
$l_1 = r_1+1$, $l_2 = r_2+1$, 
$l_3 = r_3+1$, $l_0 = n-r_1-r_2-r_3+1$.
Considering the spectrum of each side of the equation, 
applying Lemma~\ref{l:spe}, 
and denoting 
$S_{v,x}^{r_1,r_2,r_3} = {\mathrm{Sp}_C}(H_{v,x}^{r_1,r_2,r_3})$,
we get the following:
$$
S_{v,x}^{r_1,r_2,r_3} \cdot S 
=  l_1 S_{v,x}^{r_1+1,r_2-1,r_3} 
+ l_2 S_{v,x}^{r_1-1,r_2+1,r_3} 
+ l_3 S_{v,x}^{r_1,r_2,r_3+1}
+ l_0 S_{v,x}^{r_1,r_2,r_3-1}.
$$
Summarizing the last equation over all $x$ from $C_j$, we find
\begin{eqnarray}
W_{v,j}^{r_1,r_2,r_3} \cdot S  
&=&  l_1 W_{v,j}^{r_1+1,r_2-1,r_3} 
+ l_2 W_{v,j}^{r_1-1,r_2+1,r_3} 
+ l_3 \underline{W_{v,j}^{r_1,r_2,r_3+1}}
+ l_0 W_{v,j}^{r_1,r_2,r_3-1} \nonumber\\ \label{eq:vj}
\end{eqnarray}
where 
$W_{v,j}^{r_1,r_2,r_3} = (W_{v,j1}^{r_1,r_2,r_3},\ldots,W_{v,jm}^{r_1,r_2,r_3})$
and 
$W_{v,jk}^{r_1,r_2,r_3}$ 
denotes the number of the pairs $(x,y)$ 
of vertices such that
$d(v,y) = r_1+r_3$, 
$d(x,y)=r_2+r_3$, 
$d(v,x)=r_1+r_2$, 
$x\in C_j$, 
$y \in C_k$.
Because of the strong distance invariance, 
$W_{v,jk}^{r_1,r_2,r_3}$ depends on $i$ such that $v\in C_i$
and does not depend on the choice of $v$ from $C_i$.
That is, $W_{v,jk}^{r_1,r_2,r_3} = W_{ijk}^{r_1,r_2,r_3}$,
and (\ref{eq:vj}) coincides with (\ref{eq:W3}). 
The proof is over.
\begin{remark}[an alternative proof of Theorem~\ref{th:Vas}]
Formula (\ref{eq:vj}) allows to express the values $W_{v,jk}^{r_1,r_2,r_3+1}$ 
through  $W_{v,jk}^{l_1,l_2,l_3}$ with different $l_1$, $l_2$, $l_3$ satisfying
$l_1+l_2+l_3<r_1+r_2+r_3+1$. 
Since $W_{v,jk}^{r_1,r_2,r_3}=W_{v,kj}^{r_1,r_3,r_2}$,
we can say the same about 
$W_{v,jk}^{r_1,r_2+1,r_3}$. 
As a result, we can 
calculate $W_{v,jk}^{r_1,r_2,r_3}$ 
recursively, starting from 
$W_{v,jk}^{l_1,0,0}$, $l_1=0,\ldots,n$, 
i.e., from the weight distribution 
of the partition with respect to $v$. 
But the weight distribution depends
only on $C_i$ that contains $v$ 
and does not depend on the choice of $v$.
We conclude that the same is true 
for the interweight distribution.
This gives another proof of 
Theorem~\ref{th:Vas} and makes our theory 
self-contained (well, we still use the distance invariance,
but it is clear that formulas for the weight distribution
can be obtained using the technique of Section~\ref{s:proof}).
 
\end{remark}

\section{The polynomials}\label{s:poly}
In this section, 
we derive a direct formula and the enumerator for $T^{r_1,r_2,r_3}$.
Utilizing (\ref{eq:D}) or the similarity between the systems of equations
(\ref{eq:W3})--(\ref{eq:W1}) and (\ref{eq:T3})--(\ref{eq:T1}),
one can easily see that corresponding formulas for $W^{r_1,r_2,r_3}$ are
obtained by replacing $T$ by $W$ and $S'$ by $S'^{\mathrm{T}}$.

In further considerations, 
we will use the following degenerated 
but important case of equitable partitions.
By the \emph{singleton partition} 
of a graph $G=(V(G),E(G))$, 
we will mean the partition $(\{x\})_{x\in V(G)}$ 
of the vertex set into sets of cardinality one.
The singleton partition is obviously equitable,
and its quotient matrix $S$ 
coincides with the graph adjacency matrix.

\begin{lemma}\label{l:poly}
For every nonnegative integers $n$, $r_1$, $r_2$, $r_3$ 
meeting $r_1+r_2+r_3\leq n$, there is
a unique polynomial $P^{r_1,r_2,r_3}(x,y,z)$ 
of degree at most $r_1+r_2+r_3$ such that the equation
\begin{equation}\label{eq:TP}
T^{r_1,r_2,r_3}=T^{0,0,0}P^{r_1,r_2,r_3}(S',S'',S''')
\end{equation}
holds for every equitable partition of the $n$-cube 
and its quotient matrix $S$.
\end{lemma}
\begin{proof}
Since
the matrices $S'$, $S''$, $S'''$ commute with each other
(Lemma~\ref{l:commut}),
the existence of the polynomial follows 
by induction from (\ref{eq:T3})--(\ref{eq:T1}).
It remains to prove the uniqueness, which is not straightforward; 
indeed, the recursion is three-parametric,
and some values can be obtained in more than one way.

Since (\ref{eq:TP}) must hold for every equitable partition,
it is sufficient to prove the uniqueness for a fixed one. 
Let us consider the singleton partition; that is, 
$S$ is a graph adjacency matrix.

(I) We first note that the dimension 
of the vector space 
of all polynomials of degree at most $n$
in three variables 
equals $\left(n+3 \atop 3 \right)$.
This is the number of monomials of type 
$x^{r_1}y^{r_2}z^{r_3}$ of degree at most $n$,
which form a basis.

(II) Then, we see that all 
$T^{r_1,r_2,r_3}$, $r_1\geq 0$, $r_2\geq 0$, 
$r_3\geq 0$, $r_1+r_2+r_3\leq n$,
are linearly independent.
Indeed, for every vertices $i$, $j$, $k$, 
there is exactly one vector
$T^{r_1,r_2,r_3}$ 
(namely, such that the distances 
between the vertices $i$ and $j$, $j$ and $k$, $i$ and $k$ equal 
$r_1+r_2$, $r_2+r_3$, $r_1+r_3$, respectively) 
with $T^{r_1,r_2,r_3}_{ijk} \neq 0$.
The number of different non-zero $T^{r_1,r_2,r_3}$ 
is $\left(n+3 \atop 3 \right)$ again.

We can conclude from (I) and (II) that the linear map 
$P \to T^{0,0,0}P(S',S'',S''')$, from the space 
of all polynomials of degree at most $n$
in three variables 
to the vector space generated by all $T^{r_1,r_2,r_3}$,
is nonsingular.
Indeed, the dimensions of both spaces coincide, 
and the image contains a basis from 
$T^{r_1,r_2,r_3}=T^{0,0,0}P^{r_1,r_2,r_3}(S',S'',S''')$.
Hence, every $T^{r_1,r_2,r_3}$ 
is represented as $T^{0,0,0}P(S',S'',S''')$, 
where the degree of $P$ is not greater than $n$,
in only one way.
\qed\end{proof}
\begin{theorem}\label{th:gen}
The generating function 
$$
f(X,Y,Z)=\sum_{\!\!\!\!\!r_1,r_2,r_3\!\!\!\!\!}
P^{r_1,r_2,r_3}(x,y,z)X^{r_1}Y^{r_2}Z^{r_3}
$$
of the polynomials $P^{r_1,r_2,r_3}$ 
satisfying {\rm(\ref{eq:TP})} 
for every equitable partition has the form 
\begin{eqnarray*}
f(X,Y,Z)&=&
(1+X+Y+Z)^{\frac{n+x+y+z}4} 
(1+X-Y-Z)^{\frac{n+x-y-z}4} 
 \\ && \times
(1-X+Y-Z)^{\frac{n-x+y-z}4} 
(1-X-Y+Z)^{\frac{n-x-y+z}4}. 
\end{eqnarray*}
\end{theorem}
\begin{proof}
From (\ref{eq:T1}) and Lemma~\ref{l:poly}, the polynomials
$P^{r_1,r_2,r_3}$ satisfy
\begin{eqnarray} 
xP^{r_1,r_2,r_3}(x,y,z) &=& 
(r_2+1)P^{r_1,r_2+1,r_3-1}(x,y,z) 
\nonumber\\ 
&+&
(r_3+1)P^{r_1,r_2-1,r_3+1}(x,y,z) 
\nonumber\\ 
&+&
(n-r_1-r_2-r_3+1)P^{r_1-1,r_2,r_3}(x,y,z)\nonumber\\ 
&+&
(r_1+1)P^{r_1+1,r_2,r_3}(x,y,z).\label{eq:P1}
\end{eqnarray}
Multiplying by $X^{r_1}Y^{r_2}Z^{r_3}$ 
and summing over all $r_1$, $r_2$, $r_3$ 
from $0$ to $\infty$, 
we get
\begin{eqnarray*}
x \sum_{\!\!\!\!\!r_1,r_2,r_3\!\!\!\!\!}
  P^{r_1,r_2,r_3}X^{r_1}Y^{r_2}Z^{r_3} &=& 
Z\sum_{\!\!\!\!\!r_1,r_2,r_3\!\!\!\!\!}
  (r_2+1)P^{r_1,r_2+1,r_3-1}X^{r_1}Y^{r_2}Z^{r_3-1} 
  \\ &&{}+
Y\sum_{\!\!\!\!\!r_1,r_2,r_3\!\!\!\!\!}
  (r_3+1)P^{r_1,r_2-1,r_3+1}X^{r_1}Y^{r_2-1}Z^{r_3} 
  \\ &&{}+
nX \sum_{\!\!\!\!\!r_1,r_2,r_3\!\!\!\!\!}
  P^{r_1-1,r_2,r_3}(x,y,z)X^{r_1-1}Y^{r_2}Z^{r_3}
\\ &&{}- 
X^2\sum_{\!\!\!\!\!r_1,r_2,r_3\!\!\!\!\!}
  (r_1-1)P^{r_1-1,r_2,r_3}X^{r_1-2}Y^{r_2}Z^{r_3}
\\ &&{}-
 XY\sum_{\!\!\!\!\!r_1,r_2,r_3\!\!\!\!\!}
  r_2P^{r_1-1,r_2,r_3}X^{r_1-1}Y^{r_2-1}Z^{r_3}
\\&&{}-
XZ\sum_{\!\!\!\!\!r_1,r_2,r_3\!\!\!\!\!}
  r_3P^{r_1-1,r_2,r_3}X^{r_1-1}Y^{r_2}Z^{r_3-1}
\\&&{}+\sum_{\!\!\!\!\!r_1,r_2,r_3\!\!\!\!\!}
(r_1+1)P^{r_1+1,r_2,r_3}X^{r_1}Y^{r_2}Z^{r_3}.
\end{eqnarray*}
Next, denoting 
$f(X,Y,Z)=\sum_{r_1,r_2,r_3=0}^{\infty}
P^{r_1,r_2,r_3}(x,y,z)X^{r_1}Y^{r_2}Z^{r_3}
$ 
and implying $P^{r_1,r_2,r_3}(x,y,z)=0$ whenever 
$r_1<0$, $r_2<0$, or $r_3<0$, we obtain
\begin{eqnarray*}
(x-nX)f(X,Y,Z) &=& 
  (1-X^2)\frac{\partial}{\partial X}f(X,Y,Z)
  \\&&{}
+(Z-XY)\frac{\partial}{\partial Y}f(X,Y,Z)
+(Y-XZ)\frac{\partial}{\partial Z}f(X,Y,Z).
\end{eqnarray*}
The formula from the statement of the theorem 
satisfies this differential equation 
(we omit the straightforward but bulky check, 
but note that the most complicated part of the check is comparing polynomials, 
which can be verified by computer, for example, using GAP \cite{GAP}). 
In particular, this means that its Taylor coefficients satisfy 
the recursion (\ref{eq:P1}). 
Similarly, they satisfy the recursions derived from
(\ref{eq:T3}) and (\ref{eq:T2}). It remains to note that 
the coefficient at $X^0Y^0Z^0$ is $1$, as desired.
\qed\end{proof}

\begin{corollary} For all integer 
$r_1\geq 0$,
$r_2\geq 0$,
$r_3\geq 0$,
\begin{eqnarray*}
P^{r_1,r_2,r_3}(x,y,z) 
&=& \sum_{{i_1,j_1,k_1:i_1+j_1+k_1\leq r_1}
\atop{{i_2,j_2,k_2:i_2+j_2+k_2\leq r_2}
\atop{i_3,j_3,k_3:i_3+j_3+k_3\leq r_3}}}
(-1)^{i_2+i_3+j_1+j_3+k_1+k_2}
\\&&\times
\left( {\frac{\displaystyle n+x+y+z}4} \atop 
{r_1{-}i_1{-}j_1{-}k_1,\ 
 r_2{-}i_2{-}j_2{-}k_2,\ 
 r_3{-}i_3{-}j_3{-}k_3,\ \cdot}\right)
\\&&\times
\left( {\frac{n{+}x{-}y{-}z}4} \atop {i_1,i_2,i_3,\cdot}\right)
\left( {\frac{n{-}x{+}y{-}z}4} \atop {j_1,j_2,j_3,\cdot}\right)
\left( {\frac{n{-}x{-}y{+}z}4} \atop {k_1,k_2,k_3,\cdot}\right)
\end{eqnarray*}
where
$$
\left( {\delta} \atop 
  {\alpha,\beta,\gamma,\cdot}\right)
=\left( {\delta} \atop 
  {\alpha,\beta,\gamma,\delta-\alpha-\beta-\gamma}\right)
=\frac{\delta(\delta-1)\ldots
  (\delta-\alpha-\beta-\gamma+1)}{\alpha!\ \beta!\ \gamma!}.
$$
\end{corollary}
\begin{remark}
In the partial case $r_2=r_3=0$, we have
$$
P^{r,0,0}(x,y,z)=K_r(K_1^{-1}(x)),\quad 
\mbox{where }
K_r(x)=\sum_i(-1)^i\left(x \atop i\right)
             \left({n-x}\atop {r-i}\right)
$$
is the well-known Krawtchouk polynomial, 
see e.g. \cite[\S5.2]{MWS}.
So, the polynomial 
$$K^{r_1,r_2,r_3}(x,y,z) = P^{r_1,r_2,r_3}(K_1(x),K_1(y),K_1(z))$$
can be seen as a generalization of the Krawtchouk polynomial.
\end{remark}
\begin{corollary}\label{c:enum}
Given an equitable partition $C$ 
of the $n$-cube with quotient matrix $S$,
the enumerator 
$\sum_{r_1,r_2,r_3} T^{r_1,r_2,r_3}X^{r_1}Y^{r_2}Z^{r_3}$
is equal to
\begin{eqnarray}\label{eq:enum}
T^{0,0,0}
(1+X+Y+Z)^{\frac{n+S'+S''+S'''}4} 
(1+X-Y-Z)^{\frac{n+S'-S''-S'''}4}\phantom{.}
 \\ \nonumber {}\times
(1-X+Y-Z)^{\frac{n-S'+S''-S'''}4} 
(1-X-Y+Z)^{\frac{n-S'-S''+S'''}4}.
\end{eqnarray}
\end{corollary}
\begin{remark}\label{rem:alg}
The generating function $f$ in Theorem~\ref{th:gen}
is a polynomial (of degree $n$) 
if and only if all four powers are nonnegative integers.
But the quotient matrix $S$, in general, 
can have eigenvalues that make the powers negative or non-integer 
being substituted for $x$, $y$, $z$. 
This means that $P^{r_1,r_2,r_3}(S',S'',S''')$ is not necessarily
the zero matrix when $r_1+r_2+r_3 >n$. 
Nevertheless, $T^{0,0,0}P^{r_1,r_2,r_3}(S',S'',S''')$ 
will be zero in this case, 
and the enumerator in Corollary~\ref{c:enum} 
is a degree-$n$ polynomial.
This can be explained by the fact 
that the dimension of the matrix 
algebra $\mathcal{S}$ generated by $S'$, $S''$, and $S'''$ 
is higher than the dimension of its restriction 
by the action on the vector space $\mathcal T$ 
generated by $T^{r_1,r_2,r_3}$, $r_1,r_2,r_3\geq 0$.
For example, for the singleton partition,
$S$ generates an algebra with a basis $(A^{(w)})_{w=0}^n$,
such that $A^{(w)}$ and $A^{(w')}$ have no common non-zero entries provided $w\ne w'$
(this algebra is known as the Bose--Mesner algebra, 
and the matrices $A^{(w)}$ were already mentioned in Section~\ref{s:pre}).
It follows that $(A^{(w')}\otimes A^{(w'')}\otimes A^{(w''')})_{w',w'',w'''=0}^n$
is a basis of $\mathcal{S}$ as a vector space.
Hence, $\mathcal{S}$ has dimension $(n+1)^3$.
On the other hand, for any element $S_0$ of $\mathcal{S}$,
the vector $T^{0,0,0} S_0$ is a linear combination of 
$T^{0,0,0} P^{l_1,l_2,l_3}(S',S'',S''')$ with $l_1+l_2+l_3\leq n$.
Because of the commutativity, the same linear relation will be valid
if we replace $T^{0,0,0}$ by any of $T^{r_1,r_2,r_3}$ 
or by a linear combination of them.
This means that the action of any $S_0$ from $\mathcal{S}$ on $\mathcal T$ is
a linear combination of the actions of $P^{l_1,l_2,l_3}(S',S'',S''')$ with $l_1+l_2+l_3\leq n$.
That is, the space of such linear transformations of $\mathcal T$ has the dimension 
$\left( {n+3} \atop 3 \right) = \frac{(n+3)(n+2)(n+1)}6 $, 
which is smaller than $(n+1)^3$.
\end{remark}
 In the rest of this section, 
 we list the polynomials $P^{r_1,r_2,r_3}$ 
 of degree at most~$4$.
\\
$P^{0,0,0} = 1 ,$ \\ 
$P^{0,0,1} = z ,$ \\ 
$P^{0,0,2} = [z^2 - n] \,/\, 2 ,$ \\ 
$P^{0,1,1} = y z-x ,$ \\
$P^{0,0,3} = [z^3 + (2-3n) z] \,/\, 6 ,$ \\ 
$P^{0,1,2} = [y z^2 - 2 x z + (2-n) y] \,/\, 2 ,$ \\
$P^{1,1,1} = x y z - x^2 - y^2 - z^2 + 2 n,$ \\ 
$P^{0,0,4} = [z^4 + (8-6n) z^2 + (3n^2-6n)] \,/\, 24 ,$ \\
$P^{0,1,3} = [y z^3 - 3 x z^2 + (8-3n) y z + (3n-6) x] \,/\, 6 ,$ \\
$P^{0,2,2} = [y^2 z^2 - 4 x y z + 2 x^2 + (4-n) y^2 + (4-n) z^2 + (n^2-6n) ] \,/\, 4 ,$ \\
$P^{1,1,2} = [x y z^2 - 2 x^2 z - 2 y^2 z - z^3 + (6-n) x y + (5n-6) z] \,/\, 2.$ 
\begin{verbatim}
n:=Indeterminate(Rationals,1);; SetName(n,"n");
x:=Indeterminate(Rationals,2);; SetName(x,"x");
y:=Indeterminate(Rationals,3);; SetName(y,"y");
z:=Indeterminate(Rationals,4);; SetName(z,"z");
I:=Indeterminate(Rationals,5);; SetName(I,"I"); # for Id. matrix
P:=function(a,b,c) # calculates $P^{a,b,c}(x,y,z)$
 if (a<0)or(b<0)or(c<0) then return 0*I; 
 elif (a+b+c=0) then return I; 
 elif (a>0) then return ( x*P(a-1,b,c)-(n-a-b-c+2)*I*I*P(a-2,b,c)
                                      -(b+1)*I*P(a-1,b+1,c-1)
                                      -(c+1)*I*P(a-1,b-1,c+1) )/a;
 else return Value(P(b,c,a),[x,y,z],[y,z,x]); # using symmetry
 fi;
end;
Print( Value(P(1,2,3),[n,I],[6,1]), "\n" ); # example
S:=[[0,5],[3,2]];; N:=Sum(S[1]);
SI :=KroneckerProduct( S,S^0);; IIS:=KroneckerProduct(SI^0,S);;
ISI:=KroneckerProduct(S^0,SI);; SII:=KroneckerProduct(SI,S^0);;
W0:=[1,0,0,0,0,0,0,1];;
Print(W0*Value(P(1,2,3),[n,x,y,z,I],
      [5,TransposedMat(SII),ISI,IIS,SII^0]),"\n");
# Negative values mean that an equitable partition 
# with quotient matrix S does not exist
\end{verbatim}

\section{Connection with the correlation-immunity bound}\label{s:ci}
The theory of the equitable $2$-partitions 
of $n$-cubes was developed,
in its current state, by D.\,Fon-Der-Flaass in 
\cite{FDF:PerfCol,FDF:CorrImmBound,FDF:12cube.en}.
At the moment, there are three known general necessary conditions 
(\ref{eq:nabcd}), (\ref{eq:2n}), (\ref{eq:cib}) 
for a matrix
$\left({a \atop c} {b \atop d}\right)$ 
to be the quotient matrix of some equitable partition
(and only for one matrix $S=\left({1 \atop 11} {9 \atop 3}\right)$ 
satisfying these three conditions 
it is known that $S$ is not the quotient matrix 
of an equitable partition \cite{FDF:12cube.en}).
Two conditions are rather simple:
\begin{eqnarray}
&&n =a+b=c+d;  \label{eq:nabcd}\\
&&(b+c)/\mathrm{gcd}(b,c) \mbox{ divides } 2^n. \label{eq:2n}
\end{eqnarray}
Condition (\ref{eq:nabcd}) holds 
because the degree of every vertex is $n$; 
(\ref{eq:2n}) holds because $b|C_1|=c|C_2|$ 
and $|C_1|$, $|C_2|$ must be integer. 
The third condition has a nontrivial proof 
and is known as the correlation immunity bound,
because of its connection 
with the corresponding bound for Boolean functions \cite{FDF:CorrImmBound}.
\begin{theorem}[\cite{FDF:CorrImmBound}]\label{th:ci}
Assume that there exists an equitable partition 
of an $n$-cube with quotient matrix
$\left({a \atop c} {b \atop d}\right)$ where $b\neq c$; then 
\begin{equation}\label{eq:cib} 
c-a \leq n/3.
\end{equation}
\end{theorem}
Without loss of generality, we may assume
$b\geq c$ 
(otherwise we can satisfy this condition 
by renumbering the partition elements). 
The following computational result connects 
the bound (\ref{eq:cib}) with the evident
condition that the elements of the interweight distribution 
of a partition must be nonnegative.
\begin{proposition}\label{p:100}
For every integer $n$ from $1$ to $100$ 
and every integer $a$, $c>0$, $b>c$, $d$ 
satisfying {\rm (\ref{eq:nabcd})}, {\rm (\ref{eq:2n})} 
and missing {\rm (\ref{eq:cib})}, 
there are $r_2$, $r_3$ such that $T^{0,r_2,r_3}_{111}<0$, where 
$T^{\ldots}_{\ldots}$ are formally calculated using 
{\rm (\ref{eq:T3})--(\ref{eq:T1})}.
\end{proposition}

\section{The real-valued case}\label{s:real}
In this section, we will briefly discuss 
a generalization of the equitable partitions, 
where the spectrum of a single vertex 
can possess an arbitrary vector value 
over the real numbers $R$.
Let $C: V(H^n) \to R^m$ be a vector function 
whose values are $m$-tuples over $R$.
By the \emph{spectrum} $C(M)$ 
of a set $M \subseteq V(H^n)$ we will mean 
the sum of values of $C$ over $M$. 
The function $C$ is a \emph{perfect structure}
with \emph{parameter $m\times m$ matrix} $S$
if for every vertex $x$ 
the spectrum of the neighborhood of $x$ equals $C(x) S$.
The concept of perfect structures 
generalizes the equitable partitions,
whose characteristic vector functions are perfect structures.
As an example, to show that this generalization can give 
something interesting, we refer to \cite{Kro:2m-4}, 
where it is shown
that the optimal binary $1$-error-correcting codes 
of length $2^m-4$ 
are related with perfect structures, 
but not with equitable partitions in general.
Also note that the eigenfunctions of a graph are the perfect structures with $m=1$.
For this important partial case, our theory makes sense as well.

For the perfect structures, 
we can define the interweight distribution as follows.
For two (similarly, for three) $m$-tuples 
$a=(a_1,\ldots,a_m)$, $b=(b_1,\ldots,b_m)$, 
we define their tensor product $a \otimes b$ as 
$$
(c_{jk})_{j,k=1}^m
 =(c_{11},...,c_{1m},c_{21},\ldots c_{mm}), 
  \quad\mbox{where } c_{jk}=a_jb_k .
$$
For a vertex $v$, let $W_v^{r_1,r_2,r_3}$ 
denote the sum of $C(x)\otimes C(y)$ over
all pairs of vertices $(x,y)$ such that 
$d(x,y)=r_2+r_3$, $d(v,y)=r_1+r_3$, $d(v,x)=r_1+r_2$.
The value  $W^{r_1,r_2,r_3}_{v,jk}$ 
can still be treated as ``the number of the triangles
$(v,x\in C_j,y\in C_k)$ such that 
$d(x,y)=r_2+r_3$, $d(v,y)=r_1+r_3$, $d(v,x)=r_1+r_2$'' 
if the values of $C$ are treated as the multiplicities 
of a collection $(C_1,\ldots,C_m)$ of multisets. 
Similarly, let $T^{r_1,r_2,r_3}$ denote the sum of 
$C(v)\otimes C(x)\otimes C(y)$ over
all triples of vertices $(v,x,y)$ such that 
$d(x,y)=r_2+r_3$, $d(v,y)=r_1+r_3$, $d(v,x)=r_1+r_2$.
In the case of perfect structures, 
the equations similar to (\ref{eq:W3}) and (\ref{eq:W2}) hold for 
$W_v^{r_1,r_2,r_3}$, and $T^{r_1,r_2,r_3}$ satisfy
equations (\ref{eq:T3})--(\ref{eq:T1})

\begin{corollary}
Given a perfect structure $C$ and a vertex $v$, 
all the values $W_v^{r_1,r_2,r_3}$,
$r_1+r_2+r_3 \leq n$ can be calculated from 
$W^{r_1,0,0}$, $r_1 = 0,1,\ldots,n$ 
and the parameter matrix $S$.
\end{corollary}
But in general, in contrast to the case of equitable partitions,
the values $W^{r_1,0,0}$ 
cannot be calculated from $C(v)$.
\begin{corollary}
Given a perfect structure $C$, 
all the values $T^{r_1,r_2,r_3}$,
$r_1+r_2+r_3 \leq n$ can be calculated from $T^{0,0,0}$ 
and the parameter matrix $S$.
\end{corollary}
Again, $T^{0,0,0}$ is not invariant 
over all perfect structures with the same parameter matrix.

\begin{example} 
Let $n=m=2$, $C(00)=C(01)=(2,0)$, 
$C(10)=C(11)=(0,2)$,
$C'(00)=(2,0)$, $C'(01)=C'(10)=(1,1)$, $C'(11)=(0,2)$. 
Then $C$ and $C'$ are perfect structures 
with the same parameter matrix
$\left({1 \atop 1} {1 \atop 1} \right)$. 
Also, $C(v)=C'(v)$ for $v=00$.
Nevertheless, it is easy to see 
that the values $W_v^{1,0,0}$ and $T^{0,0,0}$ 
differ for $C$ and $C'$.
\end{example}
So, we can see that the strong distance invariance, 
in contrast to the distance invariance,
cannot be generalized to the perfect structures, 
while the recursive relations (\ref{eq:W3})--(\ref{eq:W2}), 
(\ref{eq:T3})--(\ref{eq:T1}) 
are still valid.
\smallskip
\section{Open problems}\label{s:concl}

For a partition of the vertex set of the $n$-cube, 
the distance invariance and the strong distance invariance are equivalent.
Indeed, the equitability of the partition follows, by definition, 
from each of these properties. 
In its turn, the equitability implies the distance invariance 
and the strong distance invariance.
The things are not so easy 
if we consider collections $C$ 
of subsets that are not partitions.
In particular, if $C$ consists of only one set.

\begin{problem}
Does there exist a distance invariant set of vertices of an $n$-cube that is not strongly distance invariant?
\end{problem}

Clearly, any set that is a cell of some equitable partition is out of consideration 
because it is strongly distance invariant.
To formulate the next question, 
recall that the \emph{distance distribution} 
of a set is the multiset of mutual distances between the set elements.

\begin{problem}
Do there exist two distance invariant sets of vertices of an $n$-cube
with the same distance distributions 
but different triangle distributions?
\end{problem}

The question can be treated as follows: is the triangle distribution 
the function of the distance distribution,
for the distance invariant sets in $n$-cubes?
The following example in the $4$-cube shows that it is not the case for arbitrary sets:
$\{0000$, $0001$, $0010$, $1111\}$, 
$\{0000$, $0001$, $0111$, $1111\}$.
The distances in both cases are $1$, $1$, $2$, $3$, $3$, $4$, 
and the triangles, in terms of $r_1$, $r_2$, $r_3$, are 
$(0,1,2)$, $(0,1,2)$, $(0,1,3)$, $(0,1,3)$ and
$(0,1,3)$, $(0,1,3)$, $(1,1,1)$, $(1,1,2)$, respectively.

Another question arises from 
the computational results in Section~\ref{s:ci}.
\begin{problem}
Explain Proposition~\ref{p:100} theoretically; 
prove it for an arbitrary $n$.
\end{problem}
In the case on $m\geq 3$, 
the interweight distributions 
can be used to obtain new nonexistence results.
\begin{example}\label{ex:22}
Consider the matrix
$S=\left(\begin{array}{ccc}
          0& 22& 0 \\ 5& 6& 11 \\ 0& 10& 12
         \end{array}
\right)$. The eigenvalues $-10$, $6$, $22$ of the matrix are eigenvalues of the $22$-cube too.
Calculating the weight distributions of a hypothetical equitable partition
with the quotient matrix $S$, we do not find a contradiction as all the elements found
are nonnegative integers.
However, calculating the triangle distribution gives a negative value for
$T^{0,8,9}_{111}$, for example. 
Hence, equitable partitions of the $22$-cube with the quotient matrix $S$ do not exist.
\end{example}

Another theoretical question is to find connections 
between the interweight distributions 
and the Terwilliger algebra of the $n$-cube \cite{Go:2002:Terw}. 
One of the relations was occasionally found by the author of the current paper
during a search in his local database of papers: 
it is the subword ``terw'', which occurs in both notations.
There should be deeper connections, 
as the Terwilliger algebra 
is related to distance triangles and its dimension 
coincides with
$\left( {n+3} \atop 3 \right)$, 
 the dimension of the algebra 
of linear transformations of $\mathcal T$
considered in Section~\ref{s:poly} (Remark~\ref{rem:alg}).
But the last algebra is commutative, while the Terwilliger algebra is not.

\begin{acknowledgements}
The author thanks Anastasiya Vasil'eva 
for useful discussions and
the anonymous referees for the work in
reviewing this manuscript.
\end{acknowledgements}
\providecommand\href[2]{#2} \providecommand\url[1]{\href{#1}{#1}}
  \def\DOI#1{{\small {DOI}:
  \href{http://dx.doi.org/#1}{#1}}}\def\DOIURL#1#2{{\small{DOI}:
  \href{http://dx.doi.org/#2}{#1}}}

\end{document}